\documentclass{amsart}
\usepackage[mathscr]{euscript}
\usepackage{amsthm}
\usepackage{amssymb}

\begin{document}

\title{One inequality inspired by Erd\H{o}s}

\author{Barbora Bat\'\i kov\'a, Tom\'a\v s J.\  Kepka and Petr C.\ N\v emec}

\address{Barbora Bat\'\i kov\'a, Department of Mathematics, CULS,
Kam\'yck\'a 129, 165 21 Praha 6 - Suchdol, Czech Republic}
\email{batikova@tf.czu.cz}

\address{Tom\'a\v s J.\ Kepka, Faculty of Education, Charles University, M.\ Rettigov\'e 4, 116 39 Praha 1, Czech Republic}
\email{kepka@karlin.mff.cuni.cz}

\address{Petr C.\ N\v emec, Department of Mathematics, CULS,
Kam\'yck\'a 129, 165 21 Praha 6 - Suchdol, Czech Republic}
\email{nemec@tf.czu.cz}

\subjclass{11D75}

\keywords{inequality}

\begin{abstract} One less common inequality in positive integers is carefully examined.
\end{abstract}

\maketitle

\newtheorem{ex}{Exercise}
\newtheorem{lemma}{Lemma}
\newtheorem{theorem}{Theorem}
\newtheorem{remark}{Remark}

In 1845, J.\ Bertrand (\cite{B}, pg.\ 129) conjectured such a statement: "For every integer $n$ greater than 6 there always exists at least one prime number contained between $n-2$ and $\frac n2$." Bertrand himself noticed that his proposition (postulate) was true for all $n$ such that $7\le n<6\cdot 10^6$ and believed it remained true for all $n\ge 7$. Seven years later, P.\ L.\ Tchebychef published a complete proof of the conjecture (\cite{T}) and various simplifications of this proof appeared through passing years. It was P.\ Erd\"os in 1932 (\cite{E}) who came up with an elegant and simple proof based on different ideas -- prime decompositions of binomial coefficients (a~thorough exposition of all this can be found in \cite{C}). The approach of Erd\"os may be considered the standard one nowadays. Two remarks could turn out to be of some interest, however. Firstly, a formally weaker form of the postulate is shown in \cite{E}. Namely, that for every positive integer $n$ there is at least one prime $p$ with $n<p\le 2n$. Of course, all odd primes $p$ are of the form $p=2n-1$, $n\ge 2$, whence it is not immediately visible why the lenient form should imply the more severe one. On the other hand, the original proof by Erd\"os can easily be completed to do the work. Secondly, the proof can be definitely viewed as elementary, but no way calculus-free. Nonetheless, we are convinced that every assertion formulated in elementary arithmetic should also have a proof as elementary as possible (e.g., it would be nice to have an elementary proof of the Last Fermat Theorem). It seems that the methods developed by Erd\"os can be transmogrified to yield demonstration(s) if not inside then at least close to Elementary Function Arithmetic -- one of the weakest fragments of arithmetic. In our (not yet quite successful) effort to find such a proof of the Bertrand postulate, we came across several unusual inequalities which may also be of independent interest. One of them is the inequality (1) below.

\medskip In the sequel, $n$ stands for a positive integer. We denote by
\begin{enumerate}
\item[(a)] $z(n)$ the largest integer such that $3z(n)<2n$;
\item[(b)] $m(n)$ the largest integer such that $m(n)^2\le 2n$;
\item[(c)] $r(n)$ the least non-negative integer such that $n\le 2^{r(n)}$;
\item[(d)] $x(n)=z(n)-(r(n)+1)m(n)$;
\end{enumerate}

Our aim is to solve (in positive integers) the inequality
\begin{equation}
z(n)-(r(n)+1)m(n)<0
\end{equation}
and find all positive integers $n$ for which the equality holds. In other words, we want to determine the sign of $x(n)$. The answer is given in Theorem \ref{T1}.

\bigskip\begin{ex}\label{R1} \rm A prospective reader is invited to check the following handy table:

\smallskip
$\begin{array}{c |r r r r r r r r r r r}
n&1&2&3&4&5&6&7&8&9&10&11\\x(n)&-1&-3&-5&-4&-9&-9&-8&-11&-15&-14&-13\end{array}$

\smallskip 
$\begin{array}{c|r r r r r r r r r c c}
n&12&13&20&50&100&200&300&400&500&1000\\x(n)&-13&-17&-23&-37&-46&-47&-41&-14&23&182\end{array}$
\end{ex}

\medskip
One observes readily that the sequences $z(n),m(n),r(n)$ are non-decreasing. For every non-negative integer $k$, $z(n)=k$ iff $3k<2n\le3k+3$ (there is one such $n$ for $k$ even and two such $n$ for $k$ odd), $m(n)=k$ iff $k^2\le2n<(k+1)^2$ (there is no such $n$ for $k=0$, $k$ such $n$ for $k$ odd and $k+1$ such $n$ for $k\ge2$ even), $r(n)=k\ge1$ iff $n\ge2$ and $2^{k-1}<n\le2^k$ (there are $2^{k-1}$ such $n$).

In order to obtain some estimates for $x(n)$, we will use the following observation.

\begin{lemma}\label{L2} Let $a,b,n$ be positive integers such that $a\le n\le b$.\newline 
{\rm(i)} $z(a)-(r(b)+1)m(b)\le x(n)\le z(b)-(r(a)+1)m(a)$.\newline
{\rm(ii)} If $x(a)<0$ and $b$ is the largest integer such that $2b\le3(r(a)+1)m(a)$ then $x(n)<0$.\newline
{\rm(iii)} If $x(b)>0$ and $a$ is the least integer with $2a\ge3(r(b)+1)m(b)+4$ then $x(n)>0$.\end{lemma}

\begin{proof} (i) The sequences $z(n),m(n),r(n)$ are non-decreasing.\newline
(ii) Since $3z(b)<2b\le3(r(a)+1)m(a)$, we have $x(n)\le z(b)-(r(a)+1)m(a)<0$ by (i).\newline
(iii) Since $3z(a)+3\ge2a\ge3(r(b)+1)m(b)+4$,  $x(n)\ge z(a)-(r(b)+1)m(b)>0$ by (i).
\end{proof}

Using this observation, we can expand negative values of $x(n)$ up and positive values down.

In order to facilitate the calculation of $x(n)$, it is useful to distinguish intervals where the values of $m(n)$ and $r(n)$ are constant (and thus the values of $x(n)$ are non-decreasing). 

\begin{lemma}\label{L3} Let $a$ be a positive integer with $r(a)=s$ and $m(a)=t$, $q$ be the largest integer such that $2q<(t
+1)^2$ and $b(a)=\min(q,2^s)$. Then $b(a)$ is the largest positive integer $k$ such that $r(k)=s$, $m(k)=t$ and if $n_1,n_2$ are positive integers such that $a\le n_1\le n_2\le b(a)$ then $x(a)\le x(n_1)=z(n_1)-(r(a)+1)m(a)\le x(n_2)\le x(b(a))$.\end{lemma}

\begin{proof} If $n>b(a)$ then either $n>2^s$ and $r(n)\ge s+1$, or $n>q$, hence $2n\ge(t+1)^2$ and $m(n)\ge t+1$. The rest follows from the fact that the sequence $z(n)$ is non-decreasing.\end{proof}

\begin{remark}\label{R3} \rm If $s=2k$ is even then $(2^k)^2=2^{2k}<2\cdot2^{2k}=2^{2k+1}<2^{2k+2}=(2k+1)^2$, and hence $2^k\le m(2^s)<2^{k+1}$. Further, by obvious induction, there is $p$ such that $2^{2k+1}=3p+2$. As $2^{2k+1}=2\cdot2^k$, $z(2^s)=p$.

If $s=2k+1$ is odd then$(2^{k+1})^2=2\cdot2^{2k+1}$, $m(2^s)=2^{k+1}$, there is $q$ such that $2^{2k+2}=3q+1$ and $z(2^s)=q$.

Now, let $n$ be such that $2^s<n \le2^{s+1}$. Then $r(n)=s+1$ and, by Lemma \ref{L2}(i), $x(n)\ge z(2^s)-(s+2)m(2^{s+1})$. If $s=2k$ then $m(2^{s+1})=2^{k+1}$ and $3x(n)\ge2^{2k+1}-2-3(s+2)2^{k+1}=2^{k+1}(2^k-6k-6)-2\ge2^{k+1}-2>0$ whenever $k\ge6$. If $k\ge6$ and $s=2k+1$ then $m(2^{s+1})<2^{k+2}$ and $3x(n)>2^{2k+2}-1-3(s+2)2^{k+2}=2^{k+2}(2^k-6k-9)-1\ge2^{k+2}-1>0$.

We have proved that $x(n)>0$ for every $n>2^{12}=4096$ ($=4^6\cdot9^0$). Further, $z(4096)=2730$, $m(4096)=90$, $r(4096)=12$ and $x(4096)=1560$.\end{remark} 

\begin{theorem}\label{T1} Let $n$ be a positive integer.\newline
{\rm(i)} $x(n)=0$ if and only if $n\in\{436,451,529,545,546\}$.\newline
{\rm(ii)} $x(n)<0$ if and only if either $n\le435$ or $n=450$ or $513\le n\le528$.\newline
{\rm(iii)} $x(n)>0$ if and only if either $437\le n\le512$, $n\ne450,451$ or $n\ge530$, $n\ne545,546$.\newline
{\rm(iv)} $x(n)>0$ for every $n\ge547$.\newline
{\rm(v)} $x(n)$ has no maximal value.\newline
{\rm(vi)} $x(129)=-59$ and $x(n)\ge-58$ for $n\ne129$.
\end{theorem}

\begin{proof} (i) -- (iv). Taking into account Remark \ref{R3}, we may assume that $n\le 4095$. From the table in Exercise \ref{R1}, it is clear that $x(n)<0$ for $1\le n\le13$. As $r(13)=4$ and $m(13)=5$, by Lemma \ref{L2}(ii) for $a=13$ we obtain $b=37$. As $x(38)=-31$, $r(38)=6$ and $m(38)=8$, Lemma \ref{L2}(ii) for $a=38$ yields $b=84$.
Using repeatedly Lemma \ref{L2}(ii) for $a=85,157,230,284,346,391,406,421$, we see that $x(n)<0$ for $1\le n\le435$ and $x(436)=0$.

 As $x(4095)=1559$, $r(4095)=12$ and $m(4095)=90$, Lemma \ref{L2}(iii) for $b=4095$ gives $a=1757$. Further, $x(1756)=462$, $r(1756)=11$ and $m(1756)=59$. By Lemma \ref{L2}(iii) for $b=1756$ we have $a=1064$.
Using repeatedly Lemma \ref{L2}(iii) for $b=1063,829,661,596,562$, we see that $x(n)>0$ for $4095\ge n\ge547$ and $x(546)=0$.

Now it remains to determine the sign of $x(n)$ for $436\le n\le 546$. We will use Lemma \ref{L3}. 
As $r(436)=9$ and $m(436)=29$, $b(436)=449$. However $x(437)=1$, and hence $x(n)>0$ for $437\le n\le449$. Further, $r(450)=9$, $m(450)=30$ and $b(450)=480$. As $x(450)=-1$, $x(451)=0$ and $x(452)=1$, $x(n)>0$ for $452\le n\le480$. 
Continuing in this way, $x(481)=10$, $r(481)=9$ and $m(481)=31$. Then $b(481)=511$ and $x(512)=21$. Thus $x(n)>0$ for $481\le n\le512$. Further, $r(513)=10$, $m(513)=32$ and $b(513)=544$. If $513\le n\le544$ then $x(n)=z(n)-352$. As the number 352 is even, there is exactly one $n$ with $z(n)=352$, namely $n=529$. Finally, $x(545)=0$.\newline
(v) Indeed, taking into account Remark \ref{R3}, if $B$ is an arbitrary positive integer and $k$ is such that $2^{k+2}>3B+1$ then $x(2^{2k+2})>B$.\newline
(vi) Using Lemma \ref{L3}, for every $i=1,2,\dots$ we define the number $k_i$ by $k_1=1$ and $k_{i+1}=b(k_i)+1$. Then, for every positive integer $i$, $r(k_i)=r(k_{i+1}-1)$, $m(k_i)=m(k_{i+1}-1)$, $r(k_{i+1})=r(k_i)+1$ or $m(k_{i+1})=m(k_i)+1$ and if $k_i\le n_1\le n_2\le k_{i+1}-1$ then $x(n_1)=z(n_1)-(r(k_i)+1)m(k_i)\le x(n_2)$. Thus $k_i\le n\le k_{i+1}-1$, $i=1,2,\dots$, are maximal intervals where $r(n),m(n)$ are constant.

The first 37 intervals $k_i\le n\le k_{i+1}-1$ are presented in the following table (the calculations, although a bit tedious, are extremely easy):

\smallskip\begin{enumerate}
\item If $n=1$ then $r(n)=0$, $m(n)=1$, $x(n)=-1$.
\item If $n=2$ then $r(n)=1$, $m(n)=2$, $x(n)=-3$.
\item If $3\le n\le4$ then $r(n)=2$, $m(n)=2$, $-5\le x(n)\le-4$.
\item If $5\le n\le7$ then $r(n)=3$, $m(n)=3$, $-9\le x(n)\le-8$.
\item If $n=8$ then $r(n)=3$, $m(n)=4$, $x(n)=-11$.
\item If $9\le n\le12$ then $r(n)=4$, $m(n)=4$, $-15\le x(n)\le-13$.
\item If $13\le n\le16$ then $r(n)=4$, $m(n)=5$, $-17\le x(n)\le-15$.
\item If $n=17$ then $r(n)=5$, $m(n)=5$, $x(n)=-19$.
\item If $18\le n\le24$ then $r(n)=5$, $m(n)=6$, $-25\le x(n)\le-21$.
\item If $25\le n\le31$ then $r(n)=5$, $m(n)=7$, $-26\le x(n)\le-22$.
\item If $n=32$ then $r(n)=5$, $m(n)=8$, $x(n)=-27$.
\item If $33\le n\le40$ then $r(n)=6$, $m(n)=8$, $-35\le x(n)\le-30$.
\item If $41\le n\le49$ then $r(n)=6$, $m(n)=9$, $-36\le x(n)\le-31$
\item If $50\le n\le60$ then $r(n)=6$, $m(n)=10$, $-37\le x(n)\le-31$.
\item If $61\le n\le64$ then $r(n)=6$, $m(n)=11$, $-37\le x(n)\le-35$.
\item If $65\le n\le71$ then $r(n)=7$, $m(n)=11$, $-45\le x(n)\le-41$.
\item If $72\le n\le84$ then $r(n)=7$, $m(n)=12$, $-49\le x(n)\le-41$.
\item If $85\le n\le97$ then $r(n)=7$, $m(n)=13$, $-48\le x(n)\le-40$.
\item If $98\le n\le112$ then $r(n)=7$, $m(n)=14$, $-47\le x(n)\le-38$.
\item If $113\le n\le127$ then $r(n)=7$, $m(n)=15$, $-45\le x(n)\le-36$.
\item If $n=128$ then $r(n)=7$, $m(n)=16$, $x(n)=-43$.
\item If $129\le n\le144$ then $r(n)=8$, $m(n)=16$, $-59\le x(n)\le-49$
\item If $145\le n\le161$ then $r(n)=8$, $m(n)=17$, $-57\le x(n)\le-46$.
\item If $162\le n\le180$ then $r(n)=8$, $m(n)=18$, $-55\le x(n)\le-43$.
\item If $181\le n\le199$ then $r(n)=8$, $m(n)=19$, $-51\le x(n)\le-39$.
\item If $200\le n\le220$ then $r(n)=8$, $m(n)=20$, $-47\le x(n)\le-34$.
\item If $221\le n\le241$ then $r(n)=8$, $m(n)=21$, $-42\le x(n)\le-29$.
\item If $242\le n\le256$ then $r(n)=8$, $m(n)=22$, $-37\le x(n)\le-28$.
\item If $257\le n\le264$ then $r(n)=9$, $m(n)=22$, $-49\le x(n)\le-45$.
\item If $265\le n\le287$ then $r(n)=9$, $m(n)=23$, $-54\le x(n)\le-39$.
\item If $288\le n\le312$ then $r(n)=9$, $m(n)=24$, $-49\le x(n)\le-33$.
\item If $313\le n\le337$ then $r(n)=9$, $m(n)=25$, $-42\le x(n)\le-26$.
\item If $338\le n\le364$ then $r(n)=9$, $m(n)=26$, $-35\le x(n)\le-18$.
\item If $365\le n\le391$ then $r(n)=9$, $m(n)=27$, $-27\le x(n)\le-10$.
\item If $392\le n\le420$ then $r(n)=9$, $m(n)=28$, $-19\le x(n)\le-1$.
\item If $421\le n\le449$ then $r(n)=9$, $m(n)=29$, $-10\le x(n)\le9$.
\item If $450\le n\le480$ then $r(n)=9$, $m(n)=30$, $-1\le x(n)\le19$.
\end{enumerate}
\smallskip
From the preceding table it is clear that  $x(n)\ge-59$ for $1\le n\le450$ and $x(129)=-59$. However $x(513)=-11$ and $b(513)=544$, and hence (ii) implies that $x(n)\ge-11$ whenever $n\ge451$. As $129=k_{22}$ and $x(130)=-58$, $x(n)\ge-58$ for $n\ne129$. \end{proof}

\begin{remark} \rm Thus 547 is the least positive integer $n$ such that $x(m)>0$ for each $m\ge n$. Notice that 547 is the onehundred-first prime number and that 101 is a~prime as well. Besides, $1+0+1=2$, $16=2^4=5+4+7$, $7=1+6$.
\end{remark}


\begin{thebibliography}{9}

\bibitem{B} 
\rm J.\ Bertrand, \it M\'emoire sur le nombre de valeurs que peut prendre une fonction quand on y permute les lettres qu'elle renferme, \rm Journal de l'Ecole Royale Polytechnique \bf 18/30 \rm(1845), 123--140.

\bibitem{C}
\rm V.\ Chv\'atal, \it The Discrete Mathematical Charms of Paul Erd\"os: A Simple Introduction, \rm Cambridge University Press, Cambridge 2021.

\bibitem{E}
\rm P.\ Erd\"os, \it Beweis eines Satzes von Tschebyschef, \rm Acta Litt.\ Sci.\ Szeged \bf 5 \rm(1932), 194--198.

\bibitem{T}
P.\ L.\ Tchebychef, \it M\'emoire sur les nombres premiers, \rm J.\ Math.\  Pures Appl. \bf 17 \rm(1852), 366--390.


\end{thebibliography}
\end{document}